\documentclass[12pt,A4]{amsart}
\usepackage{amsfonts,amsthm,amsmath}
\usepackage{amssymb}
\usepackage{rotating}
\usepackage{tikz}
\usepackage{verbatim}
\usepackage{amscd}
\usepackage[latin2]{inputenc}
\usepackage{t1enc}
\usepackage[mathscr]{eucal}
\usepackage{indentfirst}
\usepackage{graphicx}
\usepackage{graphics}
\usepackage{pict2e}
\usepackage{mathrsfs}
\usepackage{enumerate}

\usepackage[pagebackref]{hyperref}
\hypersetup{colorlinks=true}

\usepackage{color}

\numberwithin{equation}{section}
\def\blue{\textcolor{blue}}
\def\red{\textcolor{red}}
\def\green{\textcolor{green}}

\theoremstyle{plain}
\newtheorem{theorem}{Theorem}[section]
\newtheorem{lemma}[theorem]{Lemma}
\newtheorem{corollary}[theorem]{Corollary}

\theoremstyle{definition}
\newtheorem{Def}[theorem]{Definition}
\newtheorem{example}[theorem]{Example}
\newtheorem{conjecture}[theorem]{Conjecture}

\newtheorem{?}[theorem]{Problem}

\newcommand{\N}{\mathbb{N}}

\newcommand\des{\mathop{\rm des}}

\def\Bax{\mathrm{Bax}}
\def\S{\mathfrak{S}}

\def\Tlp{\mathrm{Tlp}}

\def\maj{\operatorname{maj}}

\def\imaj{\operatorname{imaj}}

\def\L{\mathfrak{L}}
\def\B{\mathfrak{B}}

\def\IDB{\operatorname{IDB}}
\def\IDT{\operatorname{IDT}}
\def\DES{\operatorname{DES}}
\def\IDES{\operatorname{IDES}}
\def\DB{\operatorname{DB}}
\def\DT{\operatorname{DT}}

\makeatother

\begin{document}
\title[Dilks' bijectivity conjecture]{Proof of Dilks' bijectivity conjecture on\\ Baxter permutations}
\author[Z. Lin]{Zhicong Lin}
\address[Zhicong Lin]{Research Center for Mathematics and Interdisciplinary Sciences, Shandong University, Qingdao 266237, P.R. China}
\email{linz@sdu.edu.cn}
\author[J. Liu]{Jing Liu}
\address[Jing Liu]{Research Center for Mathematics and Interdisciplinary Sciences, Shandong University, Qingdao 266237, P.R. China}
\email{202012008@mail.sdu.edu.cn}

\date{\today}

\begin{abstract} 
Baxter permutations originally arose in studying common fixed points of two commuting continuous functions. 
In 2015, Dilks proposed a conjectured bijection between Baxter permutations and non-intersecting triples of lattice paths in terms of inverse descent bottoms, descent positions and inverse descent tops. We prove this bijectivity conjecture by investigating its connection with the  Fran\c{c}on--Viennot bijection. As a result, we obtain a permutation interpretation of the $(t,q)$-analog of the Baxter numbers
$$
\frac{1}{{n+1\brack 1}_q{n+1\brack 2}_q}\sum_{k=0}^{n-1}q^{3{k+1\choose2}}{n+1\brack k}_q{n+1\brack k+1}_q{n+1\brack k+2}_qt^k, 
$$
where ${n\brack k}_q$ denote the $q$-binomial coefficients.
\end{abstract}


\maketitle

\section{Introduction}
Baxter permutations originated from G. Baxter's study~\cite{Bax} of  fixed points for the composite of commuting functions. Let $\S_n$ be the set of all permutations of $[n]:=\{1,2,\ldots,n\}$. A permutation $\pi=\pi_1\cdots\pi_n\in\S_n$ is a {\em Baxter permutation} if it avoids the vincular patterns $2\underline{41}3$ and $3\underline{14}2$, i.e., there is no indices $1\leq i<j<j+1<k\leq n$ such that  
$$
\pi_{j+1}<\pi_i<\pi_k<\pi_j \quad\text{or}\quad \pi_{j}<\pi_k<\pi_i<\pi_{j+1}.
$$
Denote by $\Bax_n$ the set of all Baxter permutations in $\S_n$. 

By inventing a generating tree for Baxter permutations,  algebraically manipulating the recurrence relation, and then magically guessing the correct enumeration formula, Chung, Graham, Hoggatt and Kleiman~\cite{chung} proved that 
\begin{equation}\label{eq:chung}
|\Bax_n|=\frac{1}{{n+1\choose 1}{n+1\choose 2}}\sum_{k=0}^{n-1}{n+1\choose k}{n+1\choose k+1}{n+1\choose k+2}. 
\end{equation}
A bijective proof was constructed by Viennot~\cite{Vie} and a functional equation proof was provided by Bousquet-M\'elou~\cite{bo}. The number in the right-hand side of~\eqref{eq:chung} is denoted by $B_n$, which is known as the $n$-th {\em Baxter number}. Numerous other combinatorial objects have been found to be  counted by the Baxter numbers in the literature,  some of which are in bijection with Baxter permutations (see~\cite{Dil,Noy,LK,Yan} and the references therein). The main objective of this paper is to prove a bijectivity conjecture of Dilks~\cite{Dil} relating three descent-based statistics on Baxter permutations to horizontal steps of non-intersecting triples of lattice paths. 

\begin{Def}[Descent-based statistics on permutations]\label{def:desbt}
For any permutation $\pi=\pi_1 \pi_2 \ldots \pi_n \in\S_n$, define the following three fundamental statistics:
\begin{itemize}
          \item $\DES(\pi)=\{i\in[n-1]:\pi_i>\pi_{i+1}\}$, the set of all  {\em positions of the descents} of $\pi$;
          \item $\DT(\pi)=\{\pi_{i}:\pi_{i}>\pi_{i+1}\}\subseteq [2,n]$, the set of all {\em descent tops} of $\pi$;
          \item $\DB(\pi)=\{\pi_{i+1}:\pi_{i}>\pi_{i+1}\}$, the set of all {\em descent bottoms} of $\pi$.
\end{itemize}
For the sake of convenience, we set $\IDES(\pi)=\DES(\pi^{-1})$, $\IDT(\pi)=\DT(\pi^{-1})$ and $\IDB(\pi)=\DB(\pi^{-1})$. We introduce the set of {\em modified descent tops} of $\pi$ to be 
$$\widetilde\DT(\pi)=\{\pi_{i}-1:\pi_{i}>\pi_{i+1}\}\subseteq [n-1]$$
and set $\widetilde\IDT(\pi)=\widetilde\DT(\pi^{-1})$. 
\end{Def}

For a subset $S\subseteq[n]$, a lattice path of length $n$ (i.e., has $n$ steps)  confined to the quarter plane $\N^2$ using only vertical and horizontal steps is said to {\em encode $S$} if for any $1\leq i\leq n$, 
$$
i\in S\Leftrightarrow  \text{ the $i$-th step of the lattice path is horizontal}. 
$$
Given a Baxter permutation $\pi\in\Bax_n$, define $\Gamma(\pi)$ to be the triple of lattice paths, each consisting of $n-1$ steps, where 
\begin{itemize}
\item the bottom one starts at $(2,0)$ and encodes $\IDB(\pi)$;
\item the middle one starts at $(1,1)$ and encodes $\DES(\pi)$;
\item and the top one starts at $(0,2)$ and encodes $\widetilde{\IDT}(\pi)$. 
\end{itemize}

\begin{example} Let $\pi=235419786$. Then $\IDB(\pi)=\{1,3,6,7\}$, $\DES(\pi)=\{3,4,6,8\}$ and $\widetilde{\IDT}(\pi)=\{3,4,7,8\}$. The triple $\Gamma(\pi)$ of lattice paths  are drawn in Fig.~\ref{dilk:trip}. 
\end{example}

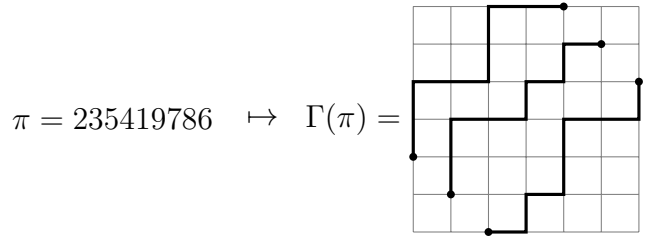
\begin{figure}
\begin{center}
\begin{tikzpicture}[scale=.5]

\draw[step=1,color=gray] (8,0) grid (14,6); 
 \draw(4,3) node{$\mapsto$};
  \draw(6.4,3) node{$\Gamma(\pi)=$};
  \draw(0,3) node{$\pi=235419786$};

\draw [very thick]
(10,0)--(11,0)--(11,1)--(12,1)--(12,3)--(14,3)--(14,4);
\draw(10,0) node{\circle*{3}};\draw(14,4) node{\circle*{3}};

\draw [very thick]
(9,1)--(9,3)--(11,3)--(11,4)--(12,4)--(12,5)--(13,5);
\draw(9,1) node{\circle*{3}};\draw(13,5) node{\circle*{3}};

\draw [very thick]
(8,2)--(8,4)--(10,4)--(10,6)--(12,6);
\draw(8,2) node{\circle*{3}};\draw(12,6) node{\circle*{3}};

\end{tikzpicture}
\end{center}
\caption{An example of the correspondence $\pi\mapsto\Gamma(\pi)$.\label{dilk:trip}}
\end{figure}

For $0\leq k\leq n-1$, let $\Bax_{n,k}:=\{\pi\in\Bax_n: \des(\pi)=k\}$, where $\des(\pi):=|\DES(\pi)|$. Denote by $\Tlp_{n,k}$ the set of all non-intersecting triples of lattice paths, each of length $n-1$ using only vertical or horizontal step, from $(0,2)$, $(1,1)$, $(2,0)$ to $(k,n-k+1)$, $(k+1,n-k)$, $(k+2,n-k-1)$. In his Ph.D. thesis, Dilks~\cite[Conjecture~3.5]{Dil} proposed  the following bijectivity conjecture. 
\begin{conjecture}[Dilks's bijectivity conjecture]\label{conj:dilks}
For $0\leq k\leq n-1$, the correspondence $\Gamma: \Bax_{n,k}\rightarrow\Tlp_{n,k}$ is a bijection. 
\end{conjecture}

We will prove Dilks' bijectivity conjecture by investigating its connection with Viennot's original  bijection~\cite{Vie} between $\Bax_{n,k}$ and $\Tlp_{n,k}$ constructed by using the  Fran\c{c}on--Viennot bijection. Dilks' bijectivity conjecture has three interesting consequences. 

For any permutation $\pi\in\S_n$, define three different major indices associated with $\DES(\pi)$, $\IDB(\pi)$ and $\widetilde\IDT(\pi)$ as
$$
\maj(\pi)=\sum_{i\in\DES(\pi)} i, \quad\imaj_B(\pi)=\sum_{i\in\IDB(\pi)} i\quad\text{and}\quad\imaj_T(\pi)=\sum_{i\in\widetilde\IDT(\pi)} i.
$$
It has already been observed by Dilks~\cite[Conjecture~3.4]{Dil} that the following permutation  interpretation of the $(t,q)$-analog of the Baxter numbers $B_n$ is a direct consequence of 
\begin{enumerate}
\item a natural bijection~\cite[Theorem~2.4]{Dil} between $\Tlp_{n,k}$ and plane partitions in a $k\times(n-1-k)\times3$ box;
\item and a $q$-counting formula for plane partitions in~\cite[Theorem~7.21.7]{St2}. 
\end{enumerate}
\begin{corollary} 
For any $n\geq1$, we have
$$
\sum_{\pi\in\Bax_n}t^{\des(\pi)}q^{\imaj_B(\pi)+\maj(\pi)+\imaj_T(\pi)}=\frac{1}{{n+1\brack 1}_q{n+1\brack 2}_q}\sum_{k=0}^{n-1}q^{3{k+1\choose2}}{n+1\brack k}_q{n+1\brack k+1}_q{n+1\brack k+2}_qt^k.
$$
\end{corollary}

A permutation $\pi\in\S_n$ is {\em alternating} if 
$$
\pi_1<\pi_2>\pi_3<\pi_4>\pi_5<\cdots
$$
and is {\em reverse alternating} if all the above inequalities are reversed. In other words, if we denote
$[n]_o$ (resp.~$[n]_e$) the set of all odd (resp.~even) integers in $[n]$, then $\pi$ is alternating if $\DES(\pi)=[n-1]_e$, while $\pi$ is reverse alternating if $\DES(\pi)=[n-1]_o$. Let 
$$
C_n:=\frac{1}{n+1}{2n\choose n}
$$
be the {\em$n$-th Catalan number} (see~\cite[pp.~219-229]{St2} for many interpretations of Catalan numbers). The following result, due to Cori,  Dulucq and Viennot~\cite{Vie2}, about enumeration of alternating Baxter permutations is a direct consequence of Conjecture~\ref{conj:dilks}.

\begin{corollary}[Cori,  Dulucq and Viennot]
The number of  (reverse) alternating Baxter permutations of length $n$ is 
\begin{equation}\label{alt:bax}
C_{\lfloor n/2\rfloor}C_{\lfloor(n+1)/2\rfloor}.
\end{equation}
\end{corollary}

The integer sequence in~\eqref{alt:bax} appears as A005817 in the OEIS~\cite{oeis}, where several other intriguing combinatorial interpretations are known. In particular, it has been proved in~\cite[Theorem~4.8]{LWZ} recently that this sequence also enumerates {\em$231$-avoiding ballot permutations}. It would be interesting to see whether there is any bijection between these two classes of pattern avoiding permutations. 

A permutation $\pi\in\S_n$ is {\em Genocchi} if 
$$
\pi_i>\pi_{i+1}\Leftrightarrow\text{ $\pi_i$ is even}.
$$
In other words, $\pi$ is a Genocchi permutation if $\DT(\pi)=[n]_e$. Genocchi permutations were introduced by Dumont~\cite{Du} to interpret the {\em Genocchi numbers}. The third consequence of Conjecture~\ref{conj:dilks} is the following new interpretation of Catalan numbers. 
\begin{corollary}
The number of permutations $\pi\in\Bax_n$ such that $\pi$ is reverse alternating and $\pi^{-1}$ is Genocchi equals the Catalan number $C_{\lfloor n/2\rfloor}$. 
\end{corollary}

\begin{example}
The five permutations  in $\Bax_6$ such that itself is reverse alternating and its inverse is Genocchi are
$$
21 436 5,\,215463,\,\, 324165,\,\,325461,\,\,435261. 
$$
\end{example}
The above interpretation of Catalan numbers  is analog to a result due to Guibert  and  Linusson~\cite{GL}, which asserts that permutations $\pi\in\Bax_n$ such that both $\pi$ and  $\pi^{-1}$ are (reverse) alternating are counted by $C_{\lfloor n/2\rfloor}$. A combinatorial bijection between these two models seems not easy. 

The rest of this paper is devoted to a proof of Dilks' bijectivity conjecture.

\section{Proof of Dilks' bijectivity conjecture}

Our starting point of the proof of  Dilks' bijectivity conjecture is the following crucial observation. 

\begin{lemma}\label{lem:inv}
If $\pi \in \Bax_n$, then $\pi^{-1}\in\Bax_n$.
\end{lemma}
\begin{proof}
Assume that $\pi^{-1}$ is not a Baxter permutation, then $\pi^{-1}$ contains  $3\underline{14}2$ pattern or $2\underline{41}3$ pattern. By the complement symmetry of these two patterns, we can assume that $\pi^{-1}$ contains the pattern $2\underline{41}3$, i.e., there exists indices $1\leq i<j<j+1<k\leq n$ such that 
$\pi^{-1}_{j+1}<\pi^{-1}_i<\pi^{-1}_k<\pi^{-1}_j$. We aim to show that $\pi$ contains the pattern $3\underline{14}2$ by induction on $l=\pi^{-1}_k-\pi^{-1}_i$, which will finish the proof of the lemma. 

If $l=1$, i.e.,  $\pi^{-1}_i=\pi^{-1}_k -1$, then the subsequence $(j+1)\,\underline{ik}\,j$ in $\pi$ forms an instance of $3\underline{14}2$ pattern. If $l>1$, then we need to consider two cases. If $\pi^{-1}_k -1$ locates before $\pi^{-1}_j$ in $\pi^{-1}$, then the subsequence $(j+1)\,\underline{m k}\,j$ in $\pi$, where the $m$-th ($m<j$) letter of $\pi^{-1}$ is $\pi^{-1}_k -1$, forms an $3\underline{14}2$ pattern and we are done. Otherwise, $\pi^{-1}_k -1$ locates after $\pi^{-1}_{j+1}$ in $\pi^{-1}$, then $\pi^{-1}_i \,\underline{\pi^{-1}_j \pi^{-1}_{j+1}}\,(\pi^{-1}_k-1)$ is still a $2\underline{41}3$ pattern in $\pi^{-1}$, but with $(\pi^{-1}_k-1)-\pi^{-1}_i=l-1<l$. This proves that $\pi$ contains the pattern $3\underline{14}2$ by induction. 
\end{proof}

In view of Lemma~\ref{lem:inv}, Conjecture~\ref{conj:dilks} is equivalent to the assertion that the correspondence  $\pi\mapsto\Gamma'(\pi)$ is a bijection between  $\Bax_{n,k}$ and $\Tlp_{n,k}$, where $\Gamma'(\pi):=\Gamma(\pi^{-1})$ is the triple of lattice paths defined by 
\begin{itemize}
\item the bottom one starts at $(2,0)$ and encodes $\DB(\pi)$;
\item the middle one starts at $(1,1)$ and encodes $\IDES(\pi)$;
\item and the top one starts at $(0,2)$ and encodes $\widetilde{\DT}(\pi)$. 
\end{itemize}
Due to the cardinality reason, it remains to show that the correspondence $\Gamma'$ is a well-defined injection.

\subsection{The correspondence $\Gamma'$ is well defined}
We will use a known recursive construction of Baxter permutations. For any $\pi\in\S_n$, a letter $\pi_i$ is a {\em left-to-right maxima} (resp.~right-to-left maxima) of $\pi$ if $\pi_j<\pi_i$ for all $j<i$ (resp.~$j>i$).
\begin{lemma}[See~\cite{chung,bo}]\label{lem:chung}
For any $\pi\in \Bax_n$, $\pi$ is obtained from some $\sigma \in \Bax_{n-1}$ by inserting $n$ into one of the following two kinds of positions:
\begin{itemize}
  \item immediately before a left-to-right maxima of $\sigma$;
  \item immediately after a right-to-left maxima of $\sigma$. 
\end{itemize}
\end{lemma}

The following result shows that $\Gamma'$ is well defined. 
\begin{lemma}
For any $\pi\in \Bax_n$, the triple of lattice paths in $\Gamma'(\pi)$ are non-intersecting. 
\end{lemma}
\begin{proof}
By Lemma~\ref{lem:chung}, $\pi$ is obtained from some $\sigma \in \Bax_{n-1}$ by inserting $n$ in one of the following three cases:
\begin{itemize}
\item The letter $n$ is inserted at the end of $\sigma$. Then 
$$
(\DB(\pi),\IDES(\pi),\widetilde\DT(\pi))=(\DB(\sigma),\IDES(\sigma),\widetilde\DT(\sigma)).
$$
\item The letter $n$ is inserted after a right-to-left maxima $\sigma_i$,  $i\neq n-1$, of $\sigma$. Then
\begin{align*}
&(\DB(\pi),\IDES(\pi))=(\DB(\sigma),\IDES(\sigma))\quad\text{and}\\
&\widetilde\DT(\pi)=(\widetilde\DT(\sigma)\setminus\{\sigma_i-1\})\cup\{n-1\}.
\end{align*}
\item The letter $n$ is inserted before a left-to-right maxima $\sigma_j$ of $\sigma$. Then
\begin{align*}
&\DB(\pi)=\DB(\sigma)\cup\{\sigma_j\},\\
&\IDES(\pi)=\IDES(\sigma)\cup\{n-1\},\\
&\widetilde\DT(\pi)=\widetilde\DT(\sigma)\cup\{n-1\}.
\end{align*}
\end{itemize}
It then follows by induction on $n$ that in either case, the triple of lattice paths in $\Gamma'(\pi)$ encoding $\DB(\pi)$, $\IDES(\pi)$ and $\widetilde\DT(\pi))$  are non-intersecting. In fact, 
\begin{itemize}
\item in the first case $\Gamma'(\pi)$ is obtained from $\Gamma'(\sigma)$ by adding a vertical step to each path; 
\item in the second case $\Gamma'(\pi)$ is obtained from $\Gamma'(\sigma)$ by adding a vertical step to the middle path and the bottom path, but changing the $(\sigma_i-1)$-th step of the top path from horizontal to vertical and then adding a horizontal step at the end; 
\item in the third case $\Gamma'(\pi)$ is obtained from $\Gamma'(\sigma)$ by adding a horizontal step to the middle path and the top path, but adding a horizontal step just after the $(\sigma_j-1)$-th step of the bottom path.
\end{itemize}
In either case, the resulting triple of lattice paths in $\Gamma'(\pi)$ are non-intersecting, which completes the proof of the lemma.
\end{proof}

\subsection{Viennot's original bijection  between  $\Bax_{n,k}$ and $\Tlp_{n,k}$}
Viennot's original bijection $\Psi$ introduced in~\cite{Vie} between $\Bax_{n,k}$ and $\Tlp_{n,k}$ consists of two main steps, the first of which is the classical  {\em Fran\c{c}on--Viennot bijection}~\cite{Fran} between permutations and Laguerre histories. The main purpose here is to prove that $\Psi$ admits a direct description using the three descent-based statistics in Definition~\ref{def:desbt}, which is key to our proof of Dilks' bijectivity conjecture. 

To begin with, let us recall the Fran\c{c}on--Viennot bijection. A {\em Motzkin path} of length $n$ is a lattice path confined to the quarter plane $\N^2$, starting from the origin, using three kinds of steps
$$
U=(1, 1) \text{ (up step)},\quad H=(1,0) \text{ (horizontal step)}\quad \text{and}\quad D= (1,-1) \text{ (down step)},
$$
and ending at $(n,0)$. A Motzkin path whose each horizontal step receives either blue (denoted a $H_b$ step) or red (denoted a $H_r$ step) is called a {\em$2$-colored Motzkin path}. Such a path is encoded as a word of length $n$ over $\{U,H_b,H_r,D\}$. A {\em Laguerre history} of length $n$ is a pair $(w,\mu)$, where $w=w_1\cdots w_n$ is a $2$-colored Motzkin path and $\mu=(\mu_1,\mu_2,\cdots,\mu_n)$ is a weight function satisfying $1\leq\mu_i\leq h_i(w)$, where 
 $$h_i(w):=1+|\{j\mid j<i, w_j=U\}|-|\{j\mid j<i, w_j=D\}|$$
  is one plus the {\em height} of the starting point of the $i$-th step of $w$. Denote by $\L_n$ the set of all Laguerre histories of length $n$. 
  
  For a permutation $\pi\in\S_n$, 
a letter $\pi_i$  is called a {\em valley} (resp.~{\em peak, double descent, double ascent}) of $\pi$ if $\pi_{i-1}>\pi_i<\pi_{i+1}$ (resp.~$\pi_{i-1}<\pi_i>\pi_{i+1}$, $\pi_{i-1}>\pi_i>\pi_{i+1}$,$\pi_{i-1}<\pi_i<\pi_{i+1}$), where    $\pi_0=\pi_{n+1}=0$ by convention.
The Fran\c{c}on--Viennot bijection $\psi_{FV}: \S_n\rightarrow\L_{n-1}$ can be defined as $\psi_{FV}(\pi)=(w,\mu)$, where for each $i\in[n-1]$: 
$$
  w_i=\left\{
  \begin{array}{ll}
  U&\mbox{if $i$ is a valley of $\pi$},  \\
 D&\mbox{if $i$ is a peak of $\pi$},  \\
   H_b&\mbox{if $i$ is a double descent of $\pi$},\\
  H_r &\mbox{if $i$ is a double ascent of $\pi$},
  \end{array}
  \right.
  $$
  and $\mu_i$ is the number of $\underline{31}2$-patterns with $i$ representing the $2$, i.e., 
$$\mu_i=(\underline{31}2)_i(\pi):=1+\#\{j: j<k\text{ and } \pi_j<\pi_k=i<\pi_{j-1}\}.$$
See Fig.~\ref{bij:vie} for an example of the bijection $\psi_{FV}$ for $\pi=512439786$. 
\begin{figure}
\begin{center}
\begin{tikzpicture}[scale=.5]
\draw[step=1,color=gray] (-1,2) grid (7,5); 
\draw[step=1,color=gray] (10,0) grid (16,6); 
 \draw(8.5,3) node{$\mapsto$};
  \draw(8.5,3.7) node{$\phi$};
 

 \green{ \draw [very thick]
 (-1,2)--(0,3) (12,0)--(13,0) (10,2)--(10,3);
 }
  \draw(12,0) node{\circle*{3}};
  
\draw(10,2) node{\circle*{3}};
 

\red{ \draw [very thick]
(0,3)--(1,3) (13,0)--(13,1) (10,3)--(10,4);
 }


\green{ \draw [very thick]
(1,3)--(2,4) (13,1)--(14,1) (10,4)--(10,5);
 }
 

\green{ \draw [very thick]
(2,4)--(3,3) (14,1)--(14,2) (10,5)--(11,5);
 }
 

\green{ \draw [very thick]
(3,3)--(4,2) (14,2)--(14,3) (11,5)--(12,5);
 }

\blue{ \draw [very thick]
(4,2)--(5,2) (14,3)--(15,3) (12,5)--(13,5);
 }
 

\green{ \draw [very thick]
(5,2)--(6,3) (15,3)--(16,3) (13,5)--(13,6);
 }

\green{ \draw [very thick]
(6,3)--(7,2) (16,3)--(16,4) (13,6)--(14,6);
 }
  \draw(16,4) node{\circle*{3}}; 
\draw(14,6) node{\circle*{3}};

\draw(11,1) node{\circle*{3}};
 \draw(-1.5,1.3) node{$\mu:$};
  \draw [dashed](11,1)--(12,0);
    \draw(-0.5,1.3) node{$1$};

 \draw(-3,3) node{$\pi\longmapsto$};  
  \draw(-2.5,3.7) node{$\psi_{FV}$};    

 \draw(0.5,1.3) node{$2$};
\draw [very thick]
(11,1)--(11,2);
 \draw [dashed](11,2)--(13,0);

 \draw(1.5,1.3) node{$2$};
\draw [very thick]
(11,2)--(11,3);
 \draw [dashed](11,3)--(13,1);

 \draw(2.5,1.3) node{$2$};
\draw [very thick]
(11,3)--(12,3);
 \draw [dashed](12,3)--(14,1);

 \draw(3.5,1.3) node{$1$};
\draw [very thick]
(12,3)--(13,3);
 \draw [dashed](13,3)--(14,2);

 \draw(4.5,1.3) node{$1$};
\draw [very thick]
(13,3)--(13,4);
 \draw [dashed](13,4)--(14,3);

 \draw(5.5,1.3) node{$1$};
\draw [very thick]
(13,4)--(14,4);
 \draw [dashed](14,4)--(15,3);

 \draw(6.5,1.3) node{$2$};
\draw [very thick]
(14,4)--(14,5);
 \draw [dashed](14,5)--(16,3);

\draw [very thick]
(14,5)--(15,5);
\draw(15,5) node{\circle*{3}};
 \draw [dashed](15,5)--(16,4);

\end{tikzpicture}
\end{center}
\caption{An example of the two steps of $\Psi=\phi\circ\psi_{FV}$ for $\pi=512439786$.\label{bij:vie}}
\end{figure}
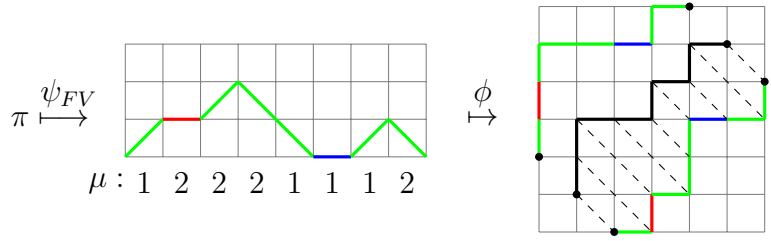

The inverse algorithm $\psi_{FV}^{-1}$ building a permutation $\pi$ (in $n$ steps) from a Laguerre history $(w,\mu)\in\L_{n-1}$ may be described  iteratively as:
\begin{itemize}
\item Initialization: $\pi=\diamond$;
\item At the $i$-th ($1\leq i\leq n-1$) step of the algorithm, replace the $\mu_i$-th $\diamond$ (from left to right) of $\pi$ by
$$
\begin{cases}
\,\diamond i\diamond& \text{if $w_i=U$},\\
\, i\diamond& \text{if $w_i=H_r$},\\
 \, i& \text{if $w_i=D$},\\
 \, \diamond i& \text{if $w_i=H_b$};
\end{cases}
$$
\item The final permutation is obtained by replacing  the last remaining $\diamond$ by $n$. 
\end{itemize}
For example, if $(w,\mu)=(UH_rUDDH_bUD,(1,2,2,2,1,1,1,2))\in\L_{8}$ is the Laguerre history in Fig.~\ref{bij:vie}, then $\pi=\psi_{FV}^{-1}(w,\mu)$ is built as follows:
\begin{align*}
\pi&=\diamond\rightarrow \diamond1\diamond\rightarrow \diamond12\diamond\rightarrow \diamond12\diamond3\diamond\rightarrow \diamond1243\diamond\rightarrow 51243\diamond\rightarrow 51243\diamond6\\
&\quad\rightarrow 51243\diamond7\diamond6\rightarrow 51243\diamond786\rightarrow 512439786. 
\end{align*}

A special Laguerre history $(w,\mu)\in\L_n$ satisfying the following two conditions is called a {\em Baxter history}:
\begin{itemize}
\item whenever $w_i=U$ or $w_i=H_b$, we have   $\mu_{i+1}=\begin{cases}\mu_i\text{ or}\\
 \mu_i+1;
\end{cases}$ 
\item whenever $w_i=D$ or $w_i=H_r$, we have   $\mu_{i+1}=\begin{cases}\mu_i\text{ or}\\
 \mu_i-1.
\end{cases}$
\end{itemize}
Denote by $\B_n$ the set of all Baxter histories  of length $n$. Viennot~\cite{Vie} proved that the Fran\c{c}on--Viennot bijection $\psi_{FV}$ restricted to a bijection between $\Bax_n$ and $\B_{n-1}$, which forms the first step of $\Psi$.

Let $\Tlp_n:=\bigcup_{k=0}^{n-1}\Tlp_{n,k}$. The second step $\phi:\B_{n-1}\rightarrow\Tlp_n$ of  $\Psi$ is defined as follows. Given a Baxter history $(w,\mu)\in\B_{n-1}$, define $\phi(w,\mu)\in\Tlp_n$ to be the triple of lattice paths such that
\begin{itemize}
\item the top one and the bottom one are determined by the $2$-colored Motzkin path $w$ by requiring that the $i$-th step of the top one (resp.~bottom one) is
\begin{enumerate}
\item vertical (resp.~horizontal) if $w_i=U$,
\item horizontal (resp.~vertical) if $w_i=D$,
\item vertical (resp.~vertical) if $w_i=H_r$,
\item horizontal (resp.~horizontal) if $w_i=H_b$;
\end{enumerate}
\item the middle one is determined by the weight function $\mu$ and the bottom one by requiring that the coordinate of the starting point of the $i$-th step of the middle one is 
$$
(x_i-\mu_i,y_i+\mu_i)\quad\text{for $1\leq i\leq n-1$}
$$
if $(x_i,y_i)$ is the coordinate of the starting point of the $i$-th step of the bottom one.
\end{itemize}
See Fig.~\ref{bij:vie} for an example of $\phi$ with $(w,\mu)=(UH_rUDDH_bUD,(1,2,2,2,1,1,1,2))\in\B_{8}$.   
 Viennot's original bijection $\Psi:\Bax_n\rightarrow\Tlp_n$ is then set to be $\phi\circ\psi_{FV}$, the functional composition of $\psi_{FV}$ and $\phi$.

 For any permutation $\pi\in\S_n$, introduce the variation of $\DT(\pi)$ as 
 $$
 \widehat{\DT}(\pi):=(\DT(\pi)\cup\{\pi_n\})\setminus\{n\}. 
 $$
 Then  Viennot's original bijection $\Psi$ admits the following direct description. 
 \begin{lemma}\label{dir:vie}
 For a given $\pi\in\Bax_n$, write $\Psi(\pi)=(P_b,P_m,P_t)$ where $P_b$ is the bottom path, $P_m$ is the middle path and $P_t$ is the top path in $\Psi(\pi)$. Then
 \begin{enumerate}
 \item the $i$-th step of $P_b$ is horizontal iff $i\in\DB(\pi)$, i.e., $P_b$ encodes $\DB(\pi)$;
 \item the $i$-th step of $P_m$ is horizontal iff $i\in\IDES(\pi)$, i.e., $P_m$ encodes $\IDES(\pi)$;
 \item the $i$-th step of $P_t$ is horizontal iff $i\in \widehat{\DT}(\pi)$, i.e., $P_t$ encodes $ \widehat{\DT}(\pi)$.
 \end{enumerate}
 \end{lemma}
 \begin{proof}
 We will prove the three points one by one in the following. Let $\psi_{FV}(\pi)=(w,\mu)$. 
 \begin{enumerate}
 \item By the construction of $\phi$, the $i$-th step of $P_b$ is horizontal iff $w_i=U$ or $w_i=H_b$, which in turn is equivalent to $i$ is a valley or a double descent of $\pi$ according to the definition of $\psi_{FV}(\pi)$. This proves the statement in (1). 
 \item The distribution of the $i$-th step of the middle path $P_m$ and the bottom path $P_b$ have the following four possibilities (see Fig.~\ref{tree:dyck}) that will be treated separately:  
\begin{figure}
\begin{tikzpicture}[scale=0.8]
\draw[-] (0,1) to (1,1) (1,0) to (2,0);
\draw[-] (4,1.5) to (5,1.5) (5.5,0.75) to (5.5,-0.25);
\draw[-] (8.5,0) to (9.5,0) (8,0.5)to(8,1.5);
\draw[-] (12,0.5)to(12,1.5)(13,0.25)to(13,-0.75);
\node at (0,1) {\circle*{4}};
\node at (1,1) {\circle*{4}};
\node at (1,0) {\circle*{4}};
\node at (2,0) {\circle*{4}};
\node at (4,1.5) {\circle*{4}};
\node at (5,1.5) {\circle*{4}};
\node at (5.5,0.75) {\circle*{4}};
\node at (5.5,-0.25) {\circle*{4}};
\node at (8.5,0){\circle*{4}};
\node at (9.5,0){\circle*{4}};
\node at (8,0.5) {\circle*{4}};
\node at (8,1.5) {\circle*{4}};
\node at (12,0.5) {\circle*{4}};
\node at (12,1.5) {\circle*{4}};
\node at (13,0.25) {\circle*{4}};
\node at (13,-0.75) {\circle*{4}};
\node at (1,-1) {$(a)$};
\node at (4.8,-1) {$(b)$};
\node at (8.5,-1){$(c)$};
\node at (12.2,-1){$(d)$};
\end{tikzpicture}
\caption{Four different cases.\label{tree:dyck}}
\end{figure}
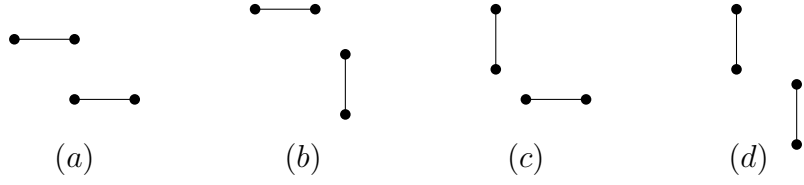

\begin{itemize}
\item Assume that the $i$-th step of $P_m$ and $P_b$ is in case $(a)$, then $i$ is a valley or a double descent in $\pi$ and $\mu$ satisfies $\mu_{i+1}= \mu_i$ (according to the definitions of the two steps of $\Psi$).  
  If $i\notin\IDES(\pi)$, namely, $i+1$ is located after $i$ in $\pi$, then $\pi$ can be written as  $\pi=\pi_1\pi_2\cdots a\,i\cdots(i+1)\cdots \pi_n$ for $a>i+1$. Thus, the subsequence $\underline{ai}(i+1)$ is a $\underline{31}2$-pattern for $\pi$ with  $i+1$ representing $2$, which forces $\mu_{i+1}\geq \mu_i+1$. A contradiction with $\mu_{i+1}= \mu_i$, which proves that $i\in\IDES(\pi)$

\item Assume that the $i$-th step of $P_m$ and $P_b$ is in case $(b)$, then $i$ is a peak or a double ascent in $\pi$ and $\mu$ satisfies $\mu_{i+1}=\mu_i-1$.  In this case, we need to show that $i\in\IDES(\pi)$. If not, then $i+1$ is located after $i$  and so  $\pi=\pi_1\pi_2\cdots c\,i\cdots(i+1)\cdots \pi_n$ for $c<i$, which forces $\mu_{i+1}\geq\mu_i$. A contradiction with $\mu_{i+1}=\mu_i-1$ and thus we have $i\in\IDES(\pi)$

\item Assume that the $i$-th step of $P_m$ and $P_b$ is in case $(c)$, then $i$ is a valley or a double descent in $\pi$ and $\mu$ satisfies $\mu_{i+1}= \mu_i+1$.  Now we need to show that $i\notin\IDES(\pi)$. If not, $i+1$ is located before $i$ in $\pi$ and so $\pi$ can be written as  $\pi=\pi_1\pi_2\cdots (i+1)\cdots\,i\cdots \pi_n$, which forces  $\mu_{i+1}\leq \mu_i$. This contradicts $\mu_{i+1}=\mu_i+1$ and so $i\notin\IDES(\pi)$. 

\item Assume that the $i$-th step of $P_m$ and $P_b$ is in case $(d)$, then $i$ is a peak or a double ascent in $\pi$ and $\mu$ satisfies $\mu_{i+1}= \mu_i$.  In this case, we need to  show that $i\notin\IDES(\pi)$. If not, $i+1$ is located before $i$ in $\pi$ and so $\pi$ can be written as  $\pi=\pi_1\pi_2\cdots (i+1)\cdots a\,i\cdots \pi_n$ with $a<i$.  Then there exists at least one $\underline{31}2$-pattern in the interval $(i+1)\cdots a\,i$  with  $i$ representing $2$, which forces  $\mu_i\geq\mu_{i+1}+1$. This contradicts $\mu_{i+1}=\mu_i$ and so $i\notin\IDES(\pi)$.
\end{itemize}
This proves the statement in (2) in all cases. 
 \item By the construction of $\phi$, the $i$-th step ($1\leq i\leq n-1$) of $P_t$ is horizontal iff $w_i=D$ or $w_i=H_b$, which in turn is equivalent to $i$ is a peak or a double descent of $\pi$ according to the definition of $\psi_{FV}(\pi)$. Since $i$ is a peak or a double descent of $\pi$ (notice that $\pi_{n+1}=0$ by convention) iff $i$ is in $\DT(\pi)\setminus\{n\}$ or $\pi_n=i$,  the statement in (3) follows. 
 \end{enumerate}
 
 The proof of the lemma is complete. 
 \end{proof}

\subsection{The correspondence $\Gamma'$ is injective}

In view of Lemma~\ref{dir:vie}, the middle and the bottom lattice paths in $\Psi(\pi)$ are in coincidence with those  in $\Gamma'(\pi)$ for any $\pi\in\Bax_n$. But the top lattice paths in $\Psi(\pi)$ and $\Gamma'(\pi)$ are different in general (compare the examples in Fig.~\ref{dilk:trip} and Fig.~\ref{bij:vie}). However, since $\Psi$ is a bijection,  Lemma~\ref{dir:vie} together with  the following lemma implies that $\Gamma'$ is injective.

\begin{lemma}\label{lem:inj}
There does not exist two Baxter permutations $\pi, \pi'\in \Bax_n$ such that
           \begin{align*}
           \widetilde{\DT}(\pi)&=\widetilde{\DT}(\pi'),\\
                     \IDES(\pi)&=\IDES(\pi'),\\
                        \DB(\pi)&=\DB(\pi'),
           \end{align*}
but $\widehat{\DT}(\pi)\neq \widehat{\DT}(\pi')$.
\end{lemma}

\begin{proof}
Note that $\widetilde{\DT}(\pi)=\widetilde{\DT}(\pi')$ is equivalent to $\DT(\pi)=\DT(\pi')$. Recall that  $
 \widehat{\DT}(\pi)=(\DT(\pi)\cup\{\pi_n\})\setminus\{n\}$. 
Assume to the contrary that such two Baxter permutations $\pi$ and $\pi'$ exist. Then $\pi_n\neq n$ and $\pi'_n\neq n$, for otherwise $\DT(\pi)=\DT(\pi')$ would imply that $\widehat{\DT}(\pi)= \widehat{\DT}(\pi')$. The only possibility is that 
\begin{equation}\label{possi}
\pi_n=i\neq\pi'_n=j \text{ for $1\leq i,j\leq n-1$ and } \widehat{\DT}(\pi)\setminus\{i\}=\widehat{\DT}(\pi')\setminus\{j\}.
\end{equation}
We aim to show that condition~\eqref{possi} could not happen when  both $\IDES(\pi)=\IDES(\pi')$ and $\DB(\pi)=\DB(\pi')$ hold.

Without loss of generality, we can assume that~\eqref{possi} holds for $i<j$. We write  $\Psi(\pi)=(P_b,P_m,P_t)$, where $P_b$ is the bottom path, $P_m$ is the middle path and $P_t$ is the top path in $\Psi(\pi)$. Similarly, we write $\Psi(\pi')=(P'_b,P'_m,P'_t)$. Set $\psi_{FV}(\pi)=(w,\mu)$ and  $\psi_{FV}(\pi')=(w',\mu')$. 
Since $\IDES(\pi)=\IDES(\pi')$ and $\DB(\pi)=\DB(\pi')$, we have $P_m=P'_m$ and $P_b=P'_b$ by Lemma~\ref{dir:vie}. It then follows from the construction of $\phi$ that the weight functions $\mu$ and $\mu'$ are equal. On the other hand, by Lemma~\ref{dir:vie} we see that condition~\eqref{possi} implies the only difference between $P_t$ and $P'_t$ are in $i$-th step and in $j$-th step. More precisely, the   $i$-th step of $P_t$ is horizontal and the $j$-th step of $P_t$ is vertical, while the  $i$-th step of $P'_t$ is vertical  and the $j$-th step of $P'_t$ is horizontal. We need to distinguish two cases according to the $i$-th step of $P_b=P'_b$ is vertical or horizontal.  Let $\pi^{(k)}$ (resp.~$\pi'^{(k)}$) be the word on $[k]\cup\{\diamond\}$ after applying the $k$-th step in the algorithm $\psi^{-1}_{FV}$ to retrieve $\pi$ (resp.~$\pi'$). 
\begin{itemize}
\item If the $i$-th step of $P_b=P'_b$ is vertical, then from the construction of $\phi$ we have $w_i=D$ and $w'_i=H_r$. As $\pi_n=i$, $\pi^{(i)}$ is obtained from $\pi^{(i-1)}$ by replacing  the rightmost $\diamond$ by $i$, while $\pi'^{(i)}$ is obtained from $\pi'^{(i-1)}=\pi^{(i-1)}$ by   replacing the rightmost $\diamond$ (since $\mu_i=\mu'_i$) by $i\diamond$. Since $\mu_k=\mu'_k$ for $k=i+1,\ldots,j-1$, the number of $\diamond$'s in $\pi'^{(j-1)}$ is one more than that in $\pi^{(j-1)}$, which force $\pi'_n\neq j$ (for otherwise,  the rightmost $\diamond$ of $\pi'^{(j-1)}$ must be replaced by $j$ or $\diamond j$, which is impossible because $\mu'_j=\mu_j$), a contradiction.

\item If the $i$-th step of $P_b=P'_b$ is horizontal, then from the construction of $\phi$ we have $w_i=H_b$ and $w'_i=U$. As $\pi_n=i$, $\pi^{(i)}$ is obtained from $\pi^{(i-1)}$ by replacing  the rightmost $\diamond$ by $\diamond i$, while $\pi'^{(i)}$ is obtained from $\pi'^{(i-1)}=\pi^{(i-1)}$ by  replacing the rightmost $\diamond$ (since $\mu_i=\mu'_i$) by $\diamond i\diamond$. Now the same reason as the first case leads to a contradiction with $\pi'_n=j$.
\end{itemize}

Since both cases lead to contradictions,  condition~\eqref{possi} could not happen and the proof of the lemma is complete. 
\end{proof}

Because of Lemma~\ref{lem:inj}, the algorithm for the inverse $(\Gamma')^{-1}$  can  be construct  as follows.  {\bf The algorithm for the inverse $(\Gamma')^{-1}$.} Given a triple of lattice path $(P_b,P_m,P_t)\in\Tlp_{n,k}$, we can retrieve the permutation $\pi=(\Gamma')^{-1}(P_b,P_m,P_t)$ according to the following two cases.
\begin{itemize}
\item If the last step of $P_t$ is vertical, then form the new top path $P'_t$ starting at $(0,2)$ of the same length that encodes $\{i+1:\text{ the $i$-th step of $P_t$ is horizontal}\}$. Define $\pi$ to be the  permutation  $\Psi^{-1}(P_b,P_m,P'_t)$. 
\item If the last step of $P_t$ is horizontal, then let 
$$
S:=\{i+1:\text{ the $i$-th step of $P_t$ is horizontal}\}\setminus\{n\}.
$$
Find the unique integer $j\in[n-1]\setminus S$ such that
\begin{enumerate}
\item the new top path $P'_t$ starting at $(0,2)$ of the same length that encodes $S\cup\{j\}$ does not intersect $P_m$;
\item and the distance between the starting point of the $j$-th step of $P_m$ and that of $P'_t$ is $\sqrt{2}$ (i.e., a diagonal unit).
\end{enumerate} 
Define $\pi$ to be the  permutation   $\Psi^{-1}(P_b,P_m,P'_t)$. 

By the construction of the algorithm  $\psi_{FV}^{-1}$ for building a permutation $\pi$  from a Laguerre history $(w,\mu)\in\L_{n-1}$, in order to guarantee that
$\pi_n=j$, we must require that $j$ is the smallest index with $w_j=D$ or $H_b$ and $\mu_j=h_j(w)$, which under $\phi$ is equivalent to the distance requirement in (2) above. The fact that $\Gamma'$ is a bijection guarantees the  existence and uniqueness  of such a $j$.
\end{itemize}

\section*{Acknowledgement} 
The authors are very grateful to  professor Viennot for explaining his bijective proof of~\eqref{eq:chung} in~\cite{Vie}. 
This work was supported
by the National Science Foundation of China grants 11871247 and the project of Qilu Young Scholars
of Shandong University.

\end{document}